\documentclass{preprint}

\usepackage{amssymb,amsmath,setspace}

\newcommand{\N}{\mathbb N}
\newcommand{\R}{{\mathbb R}}
\renewcommand{\L}{\mathrm L}
\renewcommand{\H}{\mathrm H}
\renewcommand{\(}{\left(}
\renewcommand{\)}{\right)}
\newcommand{\x}{\mathsf x}
\newcommand{\y}{\mathsf y}
\newcommand{\z}{\mathsf z}
\renewcommand{\S}{{\mathbb S^{d-1}}}

\newcommand{\ird}[1]{\int_{\R^d}{#1}\;dx}
\newcommand{\icn}[1]{\int_{\mathcal C_n}{#1}\;dy}

\newcommand{\nrm}[2]{\|{#1}\|_{\L^{#2}(\R^d)}}
\newcommand{\icnd}[1]{\int_{\mathcal C}{#1}\;dy}
\newcommand{\nrmcnd}[2]{\|{#1}\|_{\L^{#2}(\mathcal C)}}
\newcommand{\nrmcndn}[2]{\|{#1}\|_{\L^{#2}(\mathcal C_n)}}
\newcommand{\nrmrd}[2]{\|{#1}\|_{\L^{#2}(\R^d)}}
\newcommand{\finprf}{\unskip\null\hfill$\;\square$\vskip 0.3cm}
\newcommand{\C}[1]{\mathsf C_{\rm #1}}
\newcommand{\be}[1]{\begin{equation}\label{#1}}
\newcommand{\ee}{\end{equation}}

\usepackage{color}
\newcommand{\red}{}
\newcommand{\nc}{\normalcolor}

\begin{document}

\title[Existence of extremal functions]{Extremal functions for Caffarelli-Kohn-Nirenberg and logarithmic Hardy inequalities}

\author[J. Dolbeault \& M.J. Esteban]{Jean Dolbeault \& Maria J. Esteban}
\maketitle

\begin{abstract}{}\noindent
We consider a family of Caffarelli-Kohn-Nirenberg interpolation inequalities and weighted logarithmic Hardy inequalities which have been obtained recently as a limit case of the first ones. We discuss the ranges of the parameters for which the optimal constants are achieved by extremal functions. The comparison of these optimal constants with the optimal constants of Gagliardo-Nirenberg interpolation inequalities and Gross' logarithmic Sobolev inequality, both without weights, gives a general criterion for such an existence result in some particular cases.


\medskip\hspace*{-12pt}{\scriptsize\sl Keywords.\/} Sobolev spaces; Hardy-Sobolev inequality; Caffarelli-Kohn-Nirenberg inequality; logarithmic Hardy inequality; extremal functions; Kelvin transformation; Emden-Fowler transformation; radial symmetry; symmetry breaking; existence; compactness -- {\scriptsize\sl AMS classification (2000):} 49J40; 46E35; 26D10; 58E35

\end{abstract}
\section{Introduction}\label{sect1}

In this paper we discuss the existence of \emph{extremal functions} in two families of interpolation inequalities introduced in~\cite{Caffarelli-Kohn-Nirenberg-84,DDFT}: some of the Caffarelli-Kohn-Nirenberg inequalities and weighted logarithmic Hardy inequalities. By extremal functions, we mean functions for which the inequalities, written with their optimal constants, become equalities. Existence of extremal functions is a crucial issue for the study of several qualitative properties like expressions of the best constants or symmetry breaking properties of the extremal functions. Before stating our results, let us recall the two families of inequalities in which we are interested.

\subsection{Caffarelli-Kohn-Nirenberg interpolation inequalities}

Let $2^*:=\infty$ if $d=1$, $2$, and $2^*:=2\,d/(d-2)$ if $d\ge3$. Define $\vartheta(p,d):=d\,(p-2)/(2\,p)$ and consider the space $\mathcal D_a^{1,2}(\R^d)$ obtained by completion of $\mathcal D(\R^d\setminus\{0\})$ with respect to the norm $u\mapsto\nrm{\,|x|^{-a}\,\nabla u\,}2^2$. Under the restriction $\theta>1/2$ if $d=1$, notice that $\theta\in[\vartheta(p,d),1)$ for a given $p\in[2,2^*)$ if and only if $\theta\in[0,1)$, $p\in[2,p(\theta,d)]$ with $p(\theta,d):=2\,d/(d-2\,\theta)$. Let $a_c:=(d-2)/2$.
\begin{theorem}\label{Thm:CKN}{\rm \cite{Caffarelli-Kohn-Nirenberg-84}} Let $d\ge 1$. For any $p\in [2, 2^*]$ if $d\ge3$ or $p\in [2, 2^*)$ if $d=1$, $2$, for any $\theta\in[\vartheta(p,d),1]$ with $\theta>1/2$ if $d=1$, there exists a positive constant $\C{CKN}(\theta,p,a)$ such that
\be{Ineq:GenInterp}
\(\;\ird{\frac{|u|^p}{|x|^{b\,p}}}\)^\frac 2p\leq\C{CKN}(\theta,p,a)\(\;\ird{\frac{|\nabla u|^2}{|x|^{2\,a}}}\)^\theta\(\;\ird{\frac{|u|^2}{|x|^{2\,(a+1)}}}\)^{1-\theta}
\ee
for any $u\in \mathcal D^{1,2}_{a}(\R^d)$. Here $a$, $b$ and $p$ are related by $b=a-a_c+d/p$, with the restrictions $a\le b\le a+1$ if $d\ge3$, $a<b\le a+1$ if $d=2$ and $a+1/2<b\le a+1$ if $d=1$. Moreover, the constants $\C{CKN}(\theta,p,a)$ are uniformly bounded outside a neighborhood of $a=a_c$.\end{theorem}
By a transformation of Kelvin type, namely $u\mapsto |x|^{2\,(a-a_c)}\,u(x/|x|^2)$, the case $a>a_c$ can be reduced to the case $a<a_c$. See \cite{DET} for details. For simplicity of the statements, we shall therefore assume that $a<a_c$.

The case $\theta=1$, $p\in [2,2^*]$ and $d\ge 3$ has been widely discussed in the literature: see \cite{Catrina-Wang-01,DELT,DET,Felli-Schneider-03}. The case $\theta<1$ has been much less considered.

When extremal functions are radially symmetric, they are explicitly known: see \cite{DDFT,DETT}. On the other hand, \emph{symmetry breaking}, which means that extremal functions are not radially symmetric, has been established in \cite{DDFT} when $d\ge 2$ and
\[
\vartheta(p,d)\le\theta<\Theta(a,p,d)\;\mbox{ if }\;a \geq \bar a(p,d)\;\mbox{ and }\;\vartheta(p,d)\le\theta\leq 1\quad\mbox{if }\; a<\bar a(p,d)\;,
\]
with $\Theta(a,p,d):=\frac{p-2}{32\,(d-1)\,p}\, \left[ (p+2)^2\,(d^2 + 4\,a^2- 4\,a\,(d-2))-4\,p\,(p+4)\,(d-1) \right]$ and $\bar a(p,d):=\frac{d-2}{2}-{2 \sqrt{d-1}}/{\sqrt{(p-2)(p+2)}}$. This region extends the one found for $\theta=1$ in \cite{DET,Felli-Schneider-03}.

Finding extremal functions of \eqref{Ineq:GenInterp} amounts to proving the existence of minimizers for the following variational problem
\[\label{VarPb1}
\frac{1}{\C{CKN}(\theta,p,a)}= \inf_{u\in D^{1,2}_{a}(\R^d)\setminus\{0\}} \frac{\nrm{|x|^{-a}\,\nabla u}2^{2\,\theta}\,\nrm{|x|^{-(a+1)}\,u} 2^{2\,(1-\theta)}}{\nrm{|x|^{-b}\,u}p^2}
\;.
\]

\subsection{Weighted logarithmic Hardy inequalities}

In \cite{DDFT}, a new class of inequalities has been considered. These inequalities can be obtained from \eqref{Ineq:GenInterp} by taking $\theta=\gamma\,(p-2)$ and passing to the limit as $p\to2_+$.
\begin{theorem}\label{logHardy}{\rm \cite{DDFT}} Let $d\ge 1$, $a<a_c$, $\gamma \ge d/4$ and $\gamma>1/2$ if $d=2$. Then there exists a positive constant $\C{WLH}(\gamma,a)$ such that, for any $u \in \mathcal D^{1,2}_{a}(\R^d)$ normalized by $ \ird{|x|^{-2\,(a+1)}\,|u|^2}=1$, we have
\be{Ineq:GLogHardy}
\ird{\frac{|u|^2}{|x|^{2\,(a+1)}}\,\log \(|x|^{2\,(a_c-a)}\,|u|^2 \)}\leq 2\,\gamma\,\log\left[\C{WLH}(\gamma,a)\ird{\frac{|\nabla u|^2}{|x|^{2\,a}}}\right]\,.
\ee
Moreover, the constants $\C{WLH}(\gamma,a)$ are uniformly bounded outside a neighborhood of $a=a_c$.\end{theorem}
For this problem, a symmetry breaking result similar to the one of \cite{Catrina-Wang-01,Felli-Schneider-03} has been established in \cite{DDFT} for any $\gamma<1/4+(a-a_c)^2/(d-1)$ when $d\ge 2$ and $a<-1/2$.

Finding extremal functions of \eqref{Ineq:GLogHardy} amounts to proving the existence of minimizers for the following variational problem
\[\label{VarPb2}
\frac{1}{\C{WLH}(\gamma,a)}= \inf_{\substack{u\in D^{1,2}_{a}(\R^d)\\ \nrm{|x|^{-(a+1)}|u|}2\,=\,1\\}}\quad \frac{\nrm{|x|^{-a}\,\nabla u}2^2}{e^{\frac1{2\,\gamma}\ird{\frac{|u|^2}{|x|^{2\,(a+1)}}\,\log \(|x|^{2\,(a_c-a)}\,|u|^2 \)}}}
\,.
\]

\subsection{Main results}\label{Sec:Main}

Our aim is to prove existence of the extremal functions for inequalities \eqref{Ineq:GenInterp} and~\eqref{Ineq:GLogHardy}. We shall assume that $\C{CKN}(\theta,p,a)$ and $\C{WLH}(\gamma,a)$ are optimal, \emph{i.e.} take their lowest possible value. Cases of optimality among radial functions and further considerations on symmetry breaking will be dealt with in \cite{DETT}. Existence of extremal functions for \eqref{Ineq:GenInterp} has been studied in various papers in case $\theta=1$: see primarily \cite{Catrina-Wang-01} and references therein for details. In the case of radial functions, when $\theta<1$ and $d\ge 1$, existence of extremal functions has been established in \cite{DDFT} for any $\theta>\vartheta(p,d)$. Still in the radial case, similar results hold for \eqref{Ineq:GLogHardy} if $d\geq 1$ and $\gamma>1/4$. Notice that nonexistence of extremal functions has been proved in \cite{DDFT} for $d=1$ and $\theta=\vartheta(p,d)$. Nonexistence of extremal functions without symmetry assumption has also been established in \cite{Catrina-Wang-01} for $d\ge 3$, $\theta=1$ and $a=b<0$. Our main result goes as follows.
\begin{theorem}\label{MainThm} Let $d\ge 2$ and assume that $a\in(-\infty,a_c)$.
\begin{itemize}
\item[{\rm (i)}] For any $p\in(2,2^*)$ and any $\theta\in(\vartheta(p,d),1)$, \eqref{Ineq:GenInterp} admits an extremal function in $\mathcal D^{1,2}_{a}(\R^d)$. Moreover there exists a continuous function $a^*:(2,2^*)\to(-\infty,a_c)$ such that \eqref{Ineq:GenInterp} also admits an extremal function in $\mathcal D^{1,2}_{a}(\R^d)$ if $\theta=\vartheta(p,d)$ and $a\in(a^*(p),a_c)$.
\item[{\rm (ii)}] For any $\gamma>d/4$, \eqref{Ineq:GLogHardy} admits an extremal function in $\mathcal D^{1,2}_{a}(\R^d)$. Moreover there exists $\red a^{**}\nc\in(- \infty,a_c)$ such that \eqref{Ineq:GLogHardy} also admits an extremal function in $ \mathcal D^{1,2}_{a}(\R^d)$ if $\gamma=d/4$, $d\ge 3$ and $a\in(\red a^{**}\nc,a_c)$. \end{itemize}\end{theorem}
As we shall see below, the optimal constant when $p=2$, $\theta\in(0,1)$, is $(a-a_c)^{-2\,\theta}$ and it is never achieved: in this case there are no extremal functions in $\mathcal D^{1,2}_{a}(\R^d)$.

For a given $p\in(2,2^*)$, the case $\theta=\vartheta(p,d)$ deserves a more detailed analysis. Consider the following sub-family of Gagliardo-Nirenberg interpolation inequalities, which have been extensively studied in the context of nonlinear Schr\"odinger equations (see for instance~\cite{Weinstein83})\red,\nc
\[
\nrm up^2\le\C{GN}(p)\,\nrm{\nabla u}2^{2\,\vartheta(p,d)}\,\nrm u2^{2\,(1-\vartheta(p,d))}\quad\forall\; u\in\H^1(\R^d)\;.
\]
If $u$ is a radial minimizer for $1/\mathsf C_{\rm GN}(p)$, define $u_n(x):=u(x+n\,\mathsf e)$ for some $\mathsf e\in\S$. It is straightforward to check that $a\,\vartheta(p,d)+(a+1)\,(1-\vartheta(p,d))=b$. Since $\nrm{|x|^{-a}\,\nabla u_n}2\sim n^{-a}\,\nrm{\nabla u}2$, $\nrm{|x|^{-(a+1)}\,u_n}2\sim n^{-(a+1)}\,\nrm{u}2$ and $\nrm{|x|^{-b}\,u_n}p\sim n^{-b}\,\nrm up$, it follows that $\mathsf C_{\rm GN}(p)\le\C{CKN}(\vartheta(p,d),p,a)$. A more careful expansion actually shows that
\begin{multline*}
\frac 1{\C{CKN}(\vartheta(p,d),p,a)}\le\frac{\nrm{|x|^{-a}\,\nabla u_n}2^{2\,
\vartheta(p,d)}\,\nrm{|x|^{-(a+1)}\,u_n}2^{2\,(1-\vartheta(p,d))}}{\nrm{|x|^{-b}
\,u_n}p^2}\\
=\frac 1{\mathsf C_{\rm GN}(p)}\,\(1+\mathcal R\,n^{-2}+O(n^{-4})\)
\end{multline*}
as $n\to\infty$, for some real constant $\mathcal R$, that can be explicitly computed:
\be{Eqn:R}
\mathcal R=\mathcal R_1\,\frac{\nrm{|x|\,u}2^2}{\nrm u2^2}+\mathcal R_0
\ee
where $\mathcal R_0$ and $\mathcal R_1$ are polynomials of degree two in terms of $a$, with finite coefficients depending on $p$, $d$ (but not on $t$). For given $d\ge2$ and $p\in(2,2^*)$, a sufficient condition for $\mathcal R<0$ is that both $\mathcal R_1$ and $\mathcal R_0$ are negative, which defines an explicit interval in $(-\infty,a_c)$ for which we know that $\mathsf C_{\rm GN}(p)<\C{CKN}(\vartheta(p,d),p,a)$. This will be discussed in Section~\ref{Sec:R}.

Similar results can be proved for \eqref{Ineq:GLogHardy}. In that case, we shall consider Gross' logarithmic Sobolev inequality in Weissler's scale invariant form (see \cite{Gross75,MR479373})
\[
e^{\frac2{d}\ird{|u|^2\,\log |u|^2}}\le\C{LS}\,\nrm{\nabla u}2^2\quad\forall\; u\in\H^1(\R^d)\;\mbox{such that}\;\nrm u2=1\;,
\]
\red where $\C{LS}=2/(\pi\,d\,e)$. \nc With $(u_n)_n$ as above and $u(x)=(2\,\pi)^{-d/4}\,\exp(-|x|^2/4)$, we find that $\C{WLH}^{-1}\le\C{LS}^{-1}+O(n^{-2})$.

In the cases $\theta=\vartheta(p,d)$ and $\gamma=d/4$, if either $\C{CKN}(\vartheta(p,d),p,a)=C_{\rm GN}(p)$ or $\C{WLH}(d/4,a)=\C{LS}$, we have readily found a non relatively compact minimizing sequence. This indicates the possibility of non-existence of extremal functions. On the opposite, if we have strict inequalities, we can expect an existence result and this is indeed the case.
\begin{theorem}\label{MainThm2} Under the assumptions of Theorem~\ref{MainThm},
\begin{itemize}
\item[{\rm (i)}] if $\theta=\vartheta(p,d)$ and $\C{GN}(p)<\C{CKN}(\theta,p,a)$, then
\eqref{Ineq:GenInterp} admits an extremal function in $\mathcal D^{1,2}_{a}(\R^d)$,
\item[{\rm (ii)}] if $\gamma=d/4$, \red $d\ge 3$, \nc and $\C{LS}<\C{WLH}(\gamma,a)$, then
\eqref{Ineq:GLogHardy} admits an extremal function in $\mathcal D^{1,2}_{a}(\R^d)$. \red Additionnally, if $a\in(a_\star,a_c)$ with $a_\star:=a_c-\sqrt\Lambda_\star$ and $\Lambda_\star:=(d-1)\,e\,(2^{d+1}\,\pi)^{-1/(d-1)}\,\Gamma(d/2)^{2/(d-1)}$, then $\C{LS}<\C{WLH}(d/4,a)$.\nc
\end{itemize}\end{theorem}
In case (i), for $d\ge 3$ and $p=2^*$, according to \cite{Catrina-Wang-01}, it is known that $\C{GN}(2^*)=\C{CKN}(1,2^*,a)$ for any $a\le 0$. Extremal functions exist for any $a\ge 0$ and are radial, up to translations. The case $a=0$ corresponds to the celebrated extremal functions of Aubin and Talenti for Sobolev's inequality.

The criteria of Theorem~\ref{MainThm2} are sharp, in the following sense. Consider the case~(i). If for some $a_0\in(-\infty,a_c)$, \eqref{Ineq:GenInterp} admits an extremal function in $\mathcal D^{1,2}_{a}(\R^d)$ with $a=a_0$ and $\theta= \vartheta(p,d)$, then for any $a\in(a_0,a_c)$, by considering an extremal function corresponding to $a_0$ as a test function for the inequality corresponding to~$a$, we realize that $\C{CKN}(\vartheta(p,d),p,a)>\C{CKN} (\vartheta(p,d),p,a_0)$. Choose now
\[
\bar a:=\inf\{a\in(\bar a,a_c)\,:\,\C{GN}(p)<\C{CKN}(\vartheta(p,d),p,a)\}\;.
\]
If $\bar a>-\infty$, then \eqref{Ineq:GenInterp} admits an extremal function for any $a>\bar a$ and admits no extremal function for any $a<\bar a$. Similar observations hold in case (ii). See Section~\ref{Sec:R} for further comments \red and the proof of the sufficient condition for $\C{LS}<\C{WLH}(d/4,a)$.\nc

This paper is organized as follows. We shall first reformulate \eqref{Ineq:GenInterp} and \eqref{Ineq:GLogHardy} in cylindrical variables using the Emden-Fowler transformation and state some preliminary results. Sections~\ref{Sec:proofMain Theorem} and \ref{Sec:proofMain Theorem2} are devoted to the proofs of Theorems~\ref{MainThm} and~\ref{MainThm2}. In Section~\ref{Sec:R}, we shall discuss sufficient conditions for $\mathcal R$ given by \eqref{Eqn:R} to be negative and compare the results of Theorems~\ref{MainThm} (i) and \ref{MainThm2} (i) when $\theta= \vartheta(p,d)$.

\section{Observations and preliminary results}\label{Sec:Prelim}

It is very convenient to reformulate the Caffarelli-Kohn-Nirenberg inequality in cylindrical variables as in~\cite{Catrina-Wang-01}. By means of the Emden-Fowler transformation
\[
s=\log|x|\in\R\;,\quad\omega=\frac{x}{|x|}\in\S\,,\quad y=(s,\omega)\;,\quad v(y)=|x|^{a_c-a}\,u(x)\;,
\]
Inequality~\eqref{Ineq:GenInterp} for $u$ is equivalent to a Gagliardo-Nirenberg-Sobolev inequality on the cylinder $\mathcal C:=\R\times\S$:
\be{Ineq:Gen_interp_Cylinder}
\nrmcnd vp^2\leq\C{CKN}(\theta,p,a)\(\;\nrmcnd{\nabla v}2^2+\Lambda\,\nrmcnd v2^2\)^\theta\,\nrmcnd v2^{2\,(1-\theta)}\quad\forall\;v\in\H^1(\mathcal C)
\ee
with $\Lambda:=(a_c-a)^2$. Similarly, with $w(y)=|x|^{a_c-a}\,u(x)$, Inequality \eqref{Ineq:GLogHardy} is equivalent to
\[\label{Ineq:GLogHardy-w}
\icnd{|w|^2\,\log |w|^2}\leq 2\,\gamma\,\log\left[\C{WLH}(\gamma,a)\left(\nrmcnd{\nabla w}2^2+\Lambda\right)\right]
\]
for any $w\in\H^1(\mathcal C)$ such that $\nrmcnd w2=1$. When $w$ is not normalized in $\L^2(\mathcal C)$, this last inequality can also be written as
\be{Ineq:GLogHardy-w2}
\frac{\nrmcnd w2^2}{\C{WLH}(\gamma,a)}\,\exp\left[\tfrac 1{2\,\gamma}\,\icnd{\tfrac{|w|^2}{\nrmcnd w2^2}\,\log \Big(\tfrac{|w|^2}{\nrmcnd w2^2}\Big)}\right] \le \nrmcnd{\nabla v}2^2+\Lambda\,\nrmcnd v2^2\,,
\ee
for any $w\in\H^1(\mathcal C)$. We shall denote by $\C{CKN}^*(\theta,p,a)$ and $\C{WLH}^*(\gamma,a)$ the optimal constants among radial functions for \eqref{Ineq:GenInterp} and~\eqref{Ineq:GLogHardy} respectively. Radial symmetry for \eqref{Ineq:GenInterp} and \eqref{Ineq:GLogHardy} means that there are minimizers of $\mathcal E_\theta$ and $\mathcal F_\gamma$ depending only on~$s$. In such a case, $\C{CKN}(\theta,p,a)=\C{CKN}^*(\theta,p,a)$ and $\C{WLH}(\gamma,a)=\C{WLH}^*(\gamma,a)$. Radial optimal functions are explicit and the values of the optimal constants, $\C{CKN}^*(\theta,p,a)$ and $\C{WLH}^*(\gamma,a)$, have been computed in~\cite{DDFT}:
\be{Ineq:CompRad}\begin{array}{l}
\C{CKN}(\theta,p,a)\ge\C{CKN}^*(\theta,p,a)=\C{CKN}^*(\theta,p,\red a_c-1 \nc )\,\Lambda^{\frac{p-2}{2p}-\theta}\\[6pt]
\C{WLH}(\gamma,a)\ge\C{WLH}^*(\gamma,a)=\C{WLH}^*(\gamma,\red a_c-1 \nc )\,\Lambda^{-1+\frac 1{4\,\gamma}}
\end{array}\ee
where $\Lambda=(a-a_c)^2$,
\[\begin{array}{l}
\C{CKN}^*(\theta,p,\red a_c-1 \nc )\\
\hspace*{12pt}=\Big[\frac{2\,\pi^{d/2}}{\Gamma(d/2)}\Big]^{\red-\frac{p-2}p\nc}\left[\frac{(p-2)^2}{2+(2\,\theta-1)\,p}\right]^\frac{p-2}{2\,p} \left[\frac{2+(2\,\theta-1)\,p}{2\,p\,\theta}\right]^\theta \left[\frac 4{p+2}\right]^\frac{6-p}{2\,p}\left[\frac{\Gamma\left(\frac{2}{p-2}+\frac 12\right)}{\sqrt\pi\;\Gamma\left(\frac{2}{p-2}\right)}\right]^\frac{p-2}{p},
\\[6pt]
\C{WLH}^*(\gamma,\red a_c-1 \nc)= \frac{1}{4\,\gamma}\, \frac{\(4\,\gamma -1\)^{ \frac{4\,\gamma -1}{4\,\gamma}}}{(2\,\pi^{d+1}\,e)^{\frac{1}{4\,\gamma}}}\, \left[ \Gamma(d/2) \right]^{\frac{1}{2\,\gamma}}\;\mbox{if}\;\gamma>\frac 14\\
\hspace*{5cm}\mbox{and}\quad\C{WLH}^*(\tfrac 14,\red a_c-1 \nc )= \frac{ \left[ \Gamma(d/2) \right]^2}{2\,\pi^{d+1}\,e}\;\mbox{if}\;\gamma=\frac 14\;.
\end{array}\]
Symmetry breaking means that Inequalities in \eqref{Ineq:CompRad} are strict.

In case (i), for $p=2$, it is clear from \eqref{Ineq:Gen_interp_Cylinder} that $\C{CKN}(\theta,2,a)\ge\Lambda^{-\theta}$. Let $v\in\H^1(\mathcal C)$ be a function depending only on $s$ such that $\nrmcnd v2=1$ and define $v_n(y):=n^{-1}\,v(s/n)$ for any $n\ge 1$, $y=(s,\omega)\in\mathcal C$. It is therefore straightforward to observe that $\lim_{n\to\infty}(\nrmcnd{\nabla v_n}2^2+\Lambda)^\theta=\Lambda^\theta=1/\,\C{CKN}(\theta,2,a)=1/\,\C{CKN}^*(\theta,2,a)$. From \eqref{Ineq:Gen_interp_Cylinder}, we also read that equality cannot hold for a nontrivial function $v$.

\medskip The following elementary estimates will also be useful in the sequel.
\begin{lemma}\label{Lemma:algebr} For any $\x$, $\y>0$ and any $\eta\in(0,1)$, we have:
\begin{itemize}
\item[{\rm (i)}] $(1+\x)^\eta\,(1+\y)^{1-\eta}\ge 1+\x^\eta\,\y^{1-\eta}$, with strict inequality unless $\x=\y$,
\item[{\rm (ii)}] $\eta\,\x^{1/\eta}+(1-\eta)\,\y^{1/(1-\eta)}\ge \x\,\y$, with strict inequality unless $\x=\y$ and $\eta=1/2$.
\end{itemize}
\end{lemma}
\proof In case (i), let $f(\eta):= \eta\,\log(1+\x)+(1-\eta)\log\,(1+\y)-\log\(1+\x^\eta\,\y^{1-\eta}\)$. We observe that \hbox{$f(0)\!=\!f(1)\!=\!0$} and, moreover, $f(1/2)>0$, since $(1+\x)\,(1+\y)>(1+\sqrt{\x\,\y}\,)^2$ and \hbox{$(1+\x^\eta\,\y^{1-\eta})^2f''(\eta)=-\,\x^\eta\,\y^{1-\eta}(\log(\frac \x\y))^2\!<0$}, at least if $\x\neq \y$. This proves assertion (i).

In case (ii), let $f(\eta):=\eta\,\x^{1/\eta}+(1-\eta)\,\y^{1/(1-\eta)}- \x\,\y$. We observe that $f(1/2)\ge 0$ (with strict inequality unless $\x=\y$) and
\[
f''(\eta)=\eta^{-3}\,\x^{1/\eta}|\log \x|^2+(1-\eta)^{-3}\,\y^{1/(1-\eta)}|\log \y|^2>0\;,
\]
thus proving the second assertion. \finprf

The functional $w\mapsto\icnd{|w|^2\,\log\big({|w|^2}/{\| w\|_{\L^2(\Omega)}^2}\big)}$ can be seen as the limit case of $w\mapsto\nrmcnd wp^2$. At least from the point of view of H\"older's inequalities, this is indeed the case, and the following estimate will be useful in the sequel.

Let $\Omega$ be an arbitrary measurable set and consider H\"older's inequality, $\| w\|_{\L^q(\Omega)}\le\| w\|_{\L^2(\Omega)}^\eta\,\| w\|_{\L^p(\Omega)}^{1-\eta}$ with $\eta=2\,(p-q)/(q\,(p-2))$ for any $q$ such that \hbox{$2\le q\le p\le2^*$}. For $q=2$, this inequality becomes an equality, with $\eta=1$, so that we can differentiate with respect to $q$ at $q=2$ and obtain
\be{CCC45}
\int_\Omega{|w|^2\,\log\Big(\tfrac{|w|^2}{\| w\|_{\L^2(\Omega)}^2}\Big)}\le\frac p{p-2}\,\| w\|_{\L^2(\Omega)}^2\,\log\Big(\tfrac{\| w\|_{\L^p(\Omega)}^2}{\| w\|_{\L^2(\Omega)}^2}\Big)\;.
\ee

\section{Proof of Theorem \ref{MainThm}}\label{Sec:proofMain Theorem}

In case (i), let $p\in (2,2^*)$ and $\theta\in(\vartheta(p,d),1)$ . In case (ii), let $\gamma>d/4$. Consider sequences $(v_n)_n$ and $(w_n)_n$ of functions in $\H^1(\mathcal C)$, which respectively minimize the functionals
\begin{eqnarray*}
&&\mathcal E_\theta[v]:=\(\nrmcnd{\nabla v}2^2\!+\Lambda\,\nrmcnd v2^2\)^\theta\nrmcnd v2^{2\,(1-\theta)}\,,\\
&&\mathcal F_\gamma[w]:=\(\nrmcnd{\nabla w}2^2\!+\Lambda\)\,\exp\left[-\tfrac1{2\,\gamma}\icnd{|w|^2\,\log|w|^2}\right]\,,
\end{eqnarray*}
under the constraints $\nrmcnd{v_n}p=1$ and $\nrmcnd{w_n}2=1$ for any $n\in\N$. We shall first prove that these sequences are relatively compact and converge up to translations and the extraction of a subsequence towards minimizers if they are bounded in~Ê$\H^1(\mathcal C)$. Next we will establish the \emph{a priori} estimates in $\H^1(\mathcal C)$ needed for the proof of Theorem~\ref{MainThm}. Under restrictions on $a$, these \emph{a priori} estimates are also valid for $\theta=\vartheta(p,d)$ or $\gamma=d/4$ and also give an existence result for minimizers.

\subsection{Convergence of bounded minimizing sequences}\label{Sec:Convergence}

Consider $\mathcal C$ as a manifold embedded in $\R^{d+1}$ and denote by $B_r(y)$ the ball in $\R^{d+1}$ with radius $r$ centered at $y$. From \cite[Lemma 4.1]{Catrina-Wang-01}, which is an adaptation of \cite[Lemma I.1, p. 231]{MR778974}, we have the following lemma.
\begin{lemma}\label{Lem:CW}{\rm \cite{MR778974,Catrina-Wang-01}} Let $r>0$ and $q\in[2,2^*)$. If $(f_n)_n$ is bounded in $\H^1(\mathcal C)$ and if
\[
\limsup_{n\to\infty}\int_{B_r(y)\cap \mathcal C}|f_n|^q\,dy=0
\]
for any $y\in\mathcal C$, then $\lim_{n\to\infty}\nrmcnd{f_n}p=0$ for any $p\in(2,2^*)$.\end{lemma}
As a consequence of this result and of the convexity estimates of Lemma~\ref{Lemma:algebr}, \red the relative \nc convergence of bounded sequences is a rather straightforward issue.
\begin{proposition}\label{Prop:Convergence1} Let $d\ge 2$, $p\in(2,2^*)$ and $\theta\in[\vartheta(p,d),1)$. Let $(v_n)_n$ be a minimizing sequence for $\mathcal E_\theta$ such that $\nrmcnd{v_n}p=1$ for any $n\in\N$. If $(v_n)_n$ is bounded in $\H^1(\mathcal C)$, then $(v_n)_n$ is relatively compact and converges up to translations and the extraction of a subsequence to a function $v\in\H^1(\mathcal C)$ such that $\nrmcnd vp=1$ and $\mathcal E_\theta[v]=1/\,\C{CKN}(\theta,p,a)$.\end{proposition}

\begin{proof} Up to translations and the extraction of a subsequence, $(v_n)_n$ weakly converges in $\H^1(\mathcal C)$, strongly in $\L^2_{\rm loc}\cap \L^p_{\rm loc}(\mathcal C)$ and a.e. in $\mathcal C$ towards a function $v\in\H^1(\mathcal C)$. By Lemma~\ref{Lem:CW}, $v$ is non-trivial and $\nrmcnd vp\neq 0$.

Up to the extraction of subsequences, using
\[\begin{array}{l}
\lim_{n\to\infty}\nrmcnd{v_n}2^2=\nrmcnd v2^2+\lim_{n\to\infty}\nrmcnd{v_n-v}2^2\;,\\[6pt]
\lim_{n\to\infty}\nrmcnd{\nabla v_n}2^2=\nrmcnd{\nabla v}2^2+\lim_{n\to\infty}\nrmcnd{\nabla v_n-\nabla v}2^2\;,
\end{array}\]
with $\eta=\theta$,
\[ \x=\lim_{n\to\infty}\frac{\nrmcnd{\nabla v_n-\nabla v}2^2+\Lambda\,\nrmcnd{v_n-v}2^2}{\nrmcnd{\nabla v}2^2+\Lambda\,\nrmcnd v2^2}\quad\mbox{and}\quad\y=\lim_{n\to\infty}\frac{\nrmcnd{v_n-v}2^2}{\nrmcnd v2^2}\;,
\]
by Lemma \ref{Lemma:algebr} (i), we find that
\begin{multline*}
\frac 1{\C{CKN}(\theta,p,a)}=\lim_{n\to\infty}\mathcal E_\theta[v_n]\ge\mathcal E_\theta[v]+\lim_{n\to\infty}\mathcal E_\theta[v_n-v]\\
\ge\frac 1{\C{CKN}(\theta,p,a)}\(\nrmcnd vp^2+\lim_{n\to\infty}\nrmcnd{v_n-v}p^2\)
\end{multline*}
By the Brezis-Lieb Lemma (see \cite[Theorem 1]{MR699419}), we know that
\[
1=\nrmcnd{v_n}p^p=\nrmcnd vp^p +\lim_{n\to\infty}\nrmcnd{v_n-v}p^p\;.
\]
The function $f(\z):=\z^{2/p}+(1-\z)^{2/p}$ is strictly concave so that for any $\z\in[0,1]$, $f(\z)\ge 1$ with strict inequality unless $\z=0$ or $\z=1$. Applied with $\z=\nrmcnd vp^p$, this proves that $\nrmcnd vp=1$. Since $\lim_{n\to\infty}\mathcal E_\theta[v_n]\ge\mathcal E_\theta[v]$, we know that $v$ is a nontrivial extremal function for \eqref{Ineq:GenInterp}. This completes the proof. \end{proof}
\begin{proposition}\label{Prop:Convergence2} Let $d\ge 2$ and $\gamma\ge d/4$ with strict inequality if $d=2$. Let $(w_n)_n$ be a minimizing sequence for $\mathcal F_\gamma$ such that $\nrmcnd{w_n}2=1$ for any $n\in\N$. If $(w_n)_n$ is bounded in $\H^1(\mathcal C)$, then $(w_n)_n$ is relatively compact and converges up to translations and the extraction of a subsequence to a function $w\in\H^1(\mathcal C)$ such that $\nrmcnd w2=1$ and $\mathcal F_\gamma[w]=1/\,\C{WLH}(\gamma,a)$. \end{proposition}
\begin{proof} Consider now the sequence $(w_n)_n$ and denote by $w$ its weak limit in $\H^1(\mathcal C)$, after translations and the extraction of a subsequence if necessary.

By \eqref{Ineq:GLogHardy-w2} and \eqref{CCC45}, we know that
\[
\mathcal F_\gamma[w_n]\ge\C{WLH}(\gamma,a)^{-1}\,\nrm{w_n}p^{-\frac p{\gamma\,(p-2)}}
\]
for any $p\in(2,2^*)$. If $w\equiv0$, then $\lim_{n\to\infty}\nrm{w_n}p=0$ by Lemma~\ref{Lem:CW}, which contradicts the fact that $(w_n)_n$ is a minimizing sequence and $\C{WLH}(\gamma,a)$ is finite. Hence we have $\nrmcnd w2^2\neq 0$.

By the Brezis-Lieb lemma and by semi-continuity, we know that
\begin{multline*}
\frac 1{\C{WLH}(\gamma,a)}=\lim_{n\to\infty}\mathcal F_\gamma[w_n]\\
\ge\left[\nrmcnd{\nabla w}2^2\!+\Lambda\,\nrmcnd w2^2+\lim_{n\to\infty}\big(\nrmcnd{\nabla {w_n-w}}2^2\!+\Lambda\,\nrmcnd {w_n-w} 2^2\big)\right]\\
\exp\left[-\tfrac1{2\,\gamma}\lim_{n\to\infty}\icnd{|w_n|^2\,\log|w_n|^2}\right]
\end{multline*}
up to the extraction of subsequences. We may apply \eqref{Ineq:GLogHardy-w2} to $w$ and $w_n-w$. Let $\eta:=\nrmcnd w2^2$, so that $\displaystyle\lim_{n\to\infty}\nrmcnd{w_n-w}2^2=1-\eta$. We know that $\eta\in(0,1]$. If $\eta<1$, with
\[\begin{array}{l}
\x=\exp\left[\tfrac 1{2\,\gamma}\icnd{|w|^2\,\log\Big(\tfrac {|w|^2}{\nrmcnd w2^2}\Big)}\right]\,,\\[12pt]
\y=\exp\left[\tfrac 1{2\,\gamma}\lim_{n\to\infty}\icnd{|w_n-w|^2\,\log\Big(\tfrac {|w_n-w|^2}{\nrmcnd{w_n-w}2^2}\Big)}\right]\,,
\end{array}\]
we can write
\[
\frac 1{\C{WLH}(\gamma,a)}\ge\(\eta\,\x^\frac 1\eta+(1-\eta)\,\y^\frac 1{1-\eta}\)\exp\left[-\tfrac 1{2\,\gamma}\,\lim_{n\to\infty}\icnd{|w_n|^2\,\log|w_n|^2}\right]\,.
\]
We may then apply Lemma \ref{Lemma:algebr} (ii) and find that
\[
\frac 1{\C{WLH}(\gamma,a)}\ge\frac{\x\,\y}{\C{WLH}(\gamma,a)}\,\exp\left[-\tfrac 1{2\,\gamma}\,\lim_{n\to\infty}\icnd{|w_n|^2\,\log|w_n|^2}\right]\,.
\]
According to \cite[Theorem 2]{MR699419}, we have
\[
\icnd{\!|w|^2\,\log|w|^2}+\lim_{n\to\infty}\icnd{\!|w_n-w|^2\,\log|w_n-w|^2}=\lim_{n\to\infty}\icnd{\!|w_n|^2\,\log|w_n|^2}
\]
and, as a consequence, it follows that
\[
\frac 1{\C{WLH}(\gamma,a)}\ge\frac 1{\C{WLH}(\gamma,a)}\,\exp\left[-\tfrac 1{2\,\gamma}\,\big(\eta\,\log\eta+(1-\eta)\,\log(1-\eta)\big)\right]
\]
This proves that $\eta=1$. Using \eqref{CCC45}, we have \hbox{$\displaystyle\lim_{n\to\infty}\icnd{|w_n-w|^2\,\log|w_n-w|^2\!}\!=\!0$}. Hence $w$ is an extremal function, which completes the proof.\end{proof}

The remainder of this section is devoted to the \emph{a priori} estimates which are needed to establish the boundedness of minimizing sequences in $\H^1(\mathcal C)$.

\subsection{A priori estimates for Caffarelli-Kohn-Nirenberg inequalities}\label{Sec:estimatesCKN}

\begin{lemma}\label{Lem:APriori1} Assume that $d\ge 2$, $p\in(2,2^*)$ and $\theta\in(\vartheta(p,d),1)$. For any $\varepsilon>0$, there exists an $A>0$ such that, for any $v\in\H^1(\mathcal C)$ with $\nrmcnd vp=1$ and $\mathcal E_\theta[v]\le\frac{1+\varepsilon}{\C{CKN}(\theta,p,a)}$, then $\|v\|_{\H^1(\mathcal C)}\le A$.

If $d\ge 3$, there exists a positive function $a_\varepsilon^*:(2,2^*)\to(-\infty,a_c)$ such that, whenever $a\in(a_\varepsilon^*(p),a_c)$, the same conclusion holds if $\theta=\vartheta(p,d)$.\end{lemma}
\par\begin{proof} By H\"older's and Sobolev's inequalities, for any $p\in[2,2^*]$, we have
\begin{multline*}
\nrmcnd vp^2\le\(\nrmcnd v{2^*}^{\vartheta(d,p)}\,\nrmcnd v2^{1-\vartheta(d,p)}\)^2\\
\le\left[\C{CKN}(1,2^*,\red a_c-1 \nc)\(\nrmcnd{\nabla v}2^2+\Lambda_0\,\nrmcnd v2^2\)\right]^{\vartheta(d,p)}\nrmcnd v2^{2\,(1-\vartheta(d,p))}
\end{multline*}
where $\Lambda_0:=a_c^2$. Let $t:=\nrmcnd{\nabla v}2^2/\nrmcnd v2^2$ and write $\Lambda=\Lambda(a)$ for brevity. Because of the condition $\mathcal E_{\theta}[v]\le(1+\varepsilon)/\,\C{CKN}(\theta,p,a)$, we have
\begin{multline*}
(t+\Lambda)^\theta=\mathcal E_{\theta}[v]\,\frac{\nrmcnd vp^2}{\nrmcnd v2^2}\le\frac{(1+\varepsilon)\,\nrmcnd vp^2}{\C{CKN}(\theta,p,a)\,\nrmcnd v2^2}\\
\le (1+\varepsilon)\,\frac{(\C{CKN}(1,2^*,\red a_c-1 \nc))^{\vartheta(d,p)}}{\C{CKN}(\theta,p,a)}\,\(t\!+\Lambda_0\)^{\vartheta(d,p)}\;.
\end{multline*}
This proves that $t$ is bounded if $\theta>\vartheta(p,d)$.

\medskip If $d\ge 3$ and $\theta=\vartheta(p,d)$, let
\[
\kappa_\varepsilon^\theta:=(1+\varepsilon)\,\C{CKN}^*(1,2^*,\red a_c-1 \nc )^\theta/(\C{CKN}^*(\theta,p,\red a_c-1 \nc ))\;.
\]
Since $\C{CKN}(1,2^*,\red a_c-1 \nc)$, the best constant corresponding to Sobolev's critical embedding $\mathcal D^{1,2}(\R^d)\hookrightarrow\L^{2^*}(\R^d)$, is achieved among radial functions, by \eqref{Ineq:CompRad} the above condition reads
\[
t+\Lambda\le\kappa_\varepsilon\,\Lambda_0^{\frac 1d-1}\,\Lambda^{1-\frac 1d}\,\big[t+\Lambda_0\big]\;,
\]
which again shows that $t$ is bounded if \red$a\in (a_\varepsilon^*(p), a_c)$\nc, for some $a_\varepsilon^*(p)$ such that $a_c-a_\varepsilon^*(p)>0$ is not too big.

Since $\nrmcnd v2^2\,(t+\Lambda)^\theta=\mathcal E_{\theta}[v]\le(1+\varepsilon)/\,\C{CKN}(\theta,p,a)$, $\nrmcnd v2$ and $\nrmcnd{\nabla v}2=t\,\nrmcnd v2$ are also bounded as soon as $t$ is bounded, thus establishing a bound in~$\H^1(\mathcal C)$.

\medskip If $d=2$, let $\theta>\vartheta(p,2)=\frac{p-2}p$. For a choice of $q$ such that $\eta=\frac{q\,(p-2)}{p\,(q-2)}<\theta$, \emph{i.e.} for $q>\frac{p\,\theta}{2-p\,(1-\theta)}$, by H\"older's inequality we have $\nrmcnd vp\le\nrmcnd vq^\eta\,\nrmcnd v2^{1-\eta}$. Hence, from
\[
(t+\Lambda)^\theta\le\frac{(1+\varepsilon)\,\nrmcnd{v}p^2}{\C{CKN}(\theta,p,a)\,\nrmcnd{v}2^2}\le (1+\varepsilon)\,\frac{(\C{CKN}(1,q,a))^\eta}{\C{CKN}(\theta,p,a)}\,\(t\!+\Lambda\)^\eta\;,
\]
we deduce that $t$ is bounded. As above, this proves that $v$ is bounded in $\H^1(\mathcal C)$.\end{proof}

\red A more careful investigation actually provides an explicit expression of $a_\varepsilon^*(p)$. We get an upper bound for $t$ if we simultaneously have
\[
\kappa_\varepsilon\,\Lambda_0^{\frac 1d-1}\,\Lambda^{1-\frac 1d}\quad\mbox{and}\quad\Lambda<\kappa_\varepsilon\,\Lambda_0^{\frac 1d-1}\,\Lambda^{1-\frac 1d}\,\Lambda_0\;,
\]
that is 
\[
\Lambda<\min\Big\{\Lambda_0\,\kappa_\varepsilon^{-\frac d{d-1}},\Lambda_0\,\kappa_\varepsilon^d\Big\}\,.
\]
Hence, for $d\ge 3$, $t$ is bounded for $\varepsilon>0$ small enough if
\be{Cdt:Interp2}
a>a_0^*(p):=a_c-a_c\,\min\Big\{\kappa_0^{-\frac d{2(d-1)}},\kappa_0^\frac d2\Big\}\,.
\ee\nc
We shall comment on this bound in Section~\ref{Sec:R}.

\medskip\noindent\emph{Proof of Theorem~\ref{MainThm} (i).} Consider a minimizing sequence $(v_n)_n$ for $\mathcal E_\theta$ such that $\nrmcnd{v_n}p=1$. For any given $\varepsilon>0$, the condition $\mathcal E_\theta[v_n]\le\frac{1+\varepsilon}{\C{CKN}(\theta,p,a)}$ is satisfied for $n$ large enough. By Lemma~\ref{Lem:APriori1}, $(v_n)_n$ is bounded in $\H^1(\mathcal C)$. By Proposition~\ref{Prop:Convergence1}, we know that it converges towards a minimizer $v\in\H^1(\mathcal C)$ with $\nrmcnd vp=1$, up to translations and the extraction of a subsequence. This concludes the proof with $a^*=a_0^*$ given by \eqref{Cdt:Interp2}.\qed

\subsection{A priori estimates for the weighted logarithmic Hardy inequalities}\label{Sec:estimatesLH}

\begin{lemma}\label{Lem:APriori3} Assume that $d\ge 2$, $\gamma\ge d/4$ and $\gamma>1/2$ if $d=2$. For any $\varepsilon>0$, there exists an $A>0$ such that, for any $w\in\H^1(\mathcal C)$ with $\nrmcnd w2=1$ and $\mathcal F_\gamma[w]\le\frac{1+\varepsilon}{\C{WLH}(\gamma,a)}$, then $\|w\|_{\H^1(\mathcal C)}\le A$.

If $d\ge 3$, there exists $a_\varepsilon^{\red**\nc}\in(-\infty,a_c)$ such that \eqref{Ineq:GLogHardy} also admits an extremal function in $\mathcal D^{1,2}_{a}(\R^d)$ if $\gamma=d/4$, $d\ge 3$ and $a\in(a_\varepsilon^{\red**\nc},a_c)$.\end{lemma}
\begin{proof} By \eqref{Ineq:Gen_interp_Cylinder} and \eqref{CCC45}, we find that for any $\alpha<a_c$ and any $p\in(2,2^*)$,
\begin{multline*}
\icnd{|w|^2\,\log\Big(\tfrac{|w|^2}{\nrmcnd w2^2}\Big)}\le\frac p{p-2}\,\nrmcnd w2^2\,\log\Big(\tfrac{\nrmcnd wp^2}{\nrmcnd w2^2}\Big)\\\le\frac p{p-2}\,\nrmcnd w2^2\,\log\big[\C{CKN}(1,p,\alpha)\,(t+\Lambda(\alpha))\big]
\end{multline*}
with $t:=\nrmcnd{\nabla w}2^2/\nrmcnd{w}2^2$ and $\Lambda(\alpha):=(\alpha-a_c)^2$. Assuming that $\nrmcnd{w}2=1$ and $\mathcal F_\gamma[w]\le (1+\varepsilon)/\,\C{WLH}(\gamma,a)$, using \eqref{Ineq:CompRad} we find that
\[
\frac{t+\Lambda}{\big[\C{CKN}(1,p,\alpha)\,(t+\Lambda(\alpha))\big]^{\frac 1{2\,\gamma}\,\frac p{p-2}}}\le\mathcal F_\gamma[w]\le\frac{1+\varepsilon}{\C{WLH}(\gamma,a)}\le\frac{(1+\varepsilon)\,\Lambda(a)^{1-\frac 1{4\,\gamma}}}{\C{WLH}^*(\gamma,\red a_c-1 \nc )}\;,
\]
which provides a bound on $w$ in $\H^1(\mathcal C)$ if one of the two following cases:
\begin{enumerate}
\item [(i)] For $d\ge3$, if either $\gamma>\frac d4$ or $\gamma=\frac d4$ and $\Lambda\in(0,a_c^2)$ is small enough (we choose $\alpha=0$, $p=2^*$ so that $\frac 1{2\,\gamma}\,\frac p{p-2}=\frac d{4\,\gamma}$\,),
\item [(ii)] If $d=2$, for all $\gamma>\frac 12$ (we choose $\alpha=-1$, and $p>\frac{4\,\gamma}{2\,\gamma-1}$\,).
\end{enumerate}
\end{proof}

\medskip\noindent\emph{Proof of Theorem~\ref{MainThm} (ii).} Consider a minimizing sequence $(w_n)_n$ for $\mathcal F_\gamma$ such that $\nrmcnd{w_n}2=1$. For any given $\varepsilon>0$, the condition $\mathcal E_\gamma[w_n]\le\frac{1+\varepsilon}{\C{WLH}(\gamma,a)}$ is satisfied for $n$ large enough. By Lemma~\ref{Lem:APriori3}, $(w_n)_n$ is bounded in $\H^1(\mathcal C)$. By Proposition~\ref{Prop:Convergence2}, we know that it converges towards a minimizer $w\in\H^1(\mathcal C)$ with $\nrmcnd w2=1$, up to translations and the extraction of a subsequence. This concludes the proof with $a^{\red**\nc}=\liminf_{\varepsilon\to 0_+}\, a_\varepsilon^{\red**\nc}$.\qed

\section{Proof of Theorem \ref {MainThm2}}\label{Sec:proofMain Theorem2}

This section is devoted to the limit cases $\theta=\vartheta(p,d)$ or $\gamma=d/4$. A sharp criterion for the existence of extremal functions for Caffarelli-Kohn-Nirenberg and weighted logarithmic Hardy inequalities is given by the comparison of their optimal constants with the optimal constants of Gagliardo-Nirenberg and Gross' logarithmic Sobolev inequalities.

As already noted in the introduction, $\C{GN}(p)\le\C{CKN}(\vartheta(p,d),p,a)$ and $\C{LS}\le\C{WLH}(d/4,a)$ for any $a\in(-\infty,a_c)$. When equality holds, compactness of minimizing sequences is lost, because of translations. Here we shall establish a compactness result for special sequences of functions made of minimizers for $\C{CKN}(\theta_n,p,a)$ and $\C{WLH}(\gamma_n,a)$ with $\theta_n>\vartheta(p,d)$, $\displaystyle\lim_{n\to\infty}\theta_n=\vartheta(p,d)$ and $\gamma_n>d/4$, \hbox{$\displaystyle\lim_{n\to\infty}\gamma_n=d/4$}.

\subsection{Compactness of sequences of extremal functions for Caffarelli- Kohn-Nirenberg inequalities approaching the limit case $\theta=\vartheta(p,d)$}\label{Sec:estimatesCKN2}

\begin{lemma}\label{Lem:APriori2} Let $d\ge 2$, $p\in(2,2^*)$ and $a<a_c$. Consider a sequence $(\theta_n)_n$ such that $\theta_n>\vartheta(p,d)$ and $\lim_{n\to\infty}\theta_n=\vartheta(p,d)$. If $(v_n)_n$ is a sequence of extremal functions for~\eqref{Ineq:Gen_interp_Cylinder} written for $\theta=\theta_n$ such that $\nrmcnd{v_n}p=1$ for any $n\in\N$, then $(v_n)_n$ is bounded in $\H^1(\mathcal C)$ if $\C{GN}(p)<\C{CKN}(\vartheta(p,d),p,a)$. In that case, $(v_n)_n$ converges, up to translations and the extraction of a subsequence, towards a minimizer $v\in\H^1(\mathcal C)$ of $\mathcal E_{\vartheta(p,d)}$, under the constraint $\nrmcnd vp=1$.\end{lemma}
\begin{proof} For brevity, let us write $\theta=\vartheta(p,d)$ and recall that $\theta_n-\theta>0$ for all $n\in\N$. Consider first a smooth, compactly supported function $v_\varepsilon$ such that $\mathcal E_{\theta}[v_\varepsilon]\leq 1/\,\C{CKN}(\theta,p,a)+\varepsilon$ and \hbox{$\nrmcnd{v_\varepsilon}p=1$}. We have
\[
\liminf_{n\to\infty}\; \frac1{\C{CKN}(\theta_n,p,a)}\leq \liminf_{n\to\infty} \mathcal E_{\theta_n}[v_\varepsilon]= \mathcal E_{\theta}[v_\varepsilon]\le\frac1{\C{CKN}(\theta,p,a)}+\varepsilon
\]
for any $\varepsilon>0$ and can pass to the limit as $\varepsilon\to0_+$. On the other hand, we know from \cite{DDFT} that $\C{CKN}(\theta_n,p,a)$ is bounded uniformly as $n\to\infty$, so that
\be{Semi-Continuity}
0<\liminf_{n\to\infty}\frac 1{\C{CKN}(\theta_n,p,a)}\le\frac 1{\C{CKN}(\vartheta(p,d),p,a)}\;.
\ee

Consider now the sequence $(v_n)_n$ of Lemma~\ref {Lem:APriori2}. With
\[
t_n:=\nrmcnd{\nabla v_n}2^2/\nrmcnd{v_n}2^2\;,
\]
the Euler-Lagrange equation satisfied by $v_n$ for each $n\in\N$ reads
\[\label{Eqn:Euler-Lagrange}
-\,\theta_n\,\Delta v_n+\big((1-\,\theta_n)\,t_n+\Lambda\big)\,v_n= {\C{CKN}(\theta_n,p,a)}^{-1}\,(t_n+\Lambda)^{1-\theta_n}\,{v_n}^{p-1}\quad\mbox{on }\;\mathcal C\;.
\]
As in \cite{Catrina-Wang-01}, using the translation invariance of~\eqref{Ineq:Gen_interp_Cylinder} in the $s$-variable, the invariance of the functional ${\mathcal E}_{\theta}$ under rotations on $\S$, and the fact that $v_n$ is a minimizer, we can assume that $v_n$ is nonnegative and achieves its maximum at some fixed, given point $\omega_*\in\S$. By the maximum principle, we know that $-\Delta v_n(0,\omega_*)\ge 0$ and hence $M_n:=v_n(0,\omega_*)=\nrmcnd{v_n}\infty$ is such that
\[
M_n^{p-2}\geq \C{CKN}(\theta_n,p,a)\,\big((1-\theta_n)\,t_n+\Lambda\big)\,(t_n+\Lambda)^{\theta_n-1}\,.
\]
After the extraction of a subsequence, we may assume that $(L_n)_n$ converges and $L:=\lim_{n\to\infty}L_n\in ((1-\theta_n)\,\C{CKN}(\theta_n,p,a), +\infty]$ where $L_n:=M_n^{p-2}\,t_n^{-\theta_n}$. Let us consider the rescaled function $f_n$ defined on $\mathcal C_n:= \R\times \sigma_n\,\S$ by $v_n(x)=M_n\, f_n(y)$, where $y=\sigma_n\,x$ and $\sigma_n^2=(t_n+\Lambda)^{1-\theta_n}\,M_n^{p-2}/\,\C{CKN}(\theta_n,p,a)$. For any $n\in\N$, the function $f_n$ is nonnegative, satisfies
\[\label{Eqn:rescaled-Euler-Lagrange}
-\,\theta_n\,\Delta f_n+\big((1-\theta_n)\,t_n+\Lambda\big)\,\sigma_n^{-2}\,f_n=f_n^{p-1}
\]
and reaches its maximum value, $1$, at the point $(0, \omega_n)$, where $\omega_n=\sigma_n\,\omega_*$.

Assume by contradiction that
\[
\lim_{n\to\infty}t_n=\infty\;.
\]
In such a case, we know that $\mathcal E_{\theta_n}[v_n]\sim t_n^{\theta_n}\,\nrm{v_n}2^2$ so that $\nrm{v_n}2^2\sim t_n^{-\theta_n}$ and $\nrm{\nabla v_n}2^2=t_n\,\nrm{v_n}2^2\sim t_n^{1-\theta_n}$. Moreover, we have $M_n^{p-2}=L_n\,t_n^{\theta_n}\to\infty$ and, by \eqref{Semi-Continuity}, $\sigma_n^2\sim t_n^{1-\theta_n}\,M_n^{p-2}\sim L_n\,t_n\to\infty$, and $f_n$ solves
\[
-\,\theta_n\,\Delta f_n+\frac{1-\theta_n}{L_n}\,(1+o(1))\,f_n=f_n^{p-1}\;.
\]
As a consequence, $\Delta f_n$ is locally uniformly bounded.

Next we define on $\R^d$ the functions $g_n(s,\Pi_n\,\omega):= f_n(s,\omega)\,\rho_n(s,\Pi_n\,\omega)$, where $\omega\in \sigma_n\,\S$ and $\Pi_n$ is the stereographic projection of $\sigma_n\,\S$ onto $\R^{d-1}$, considered as the tangent plane to $\S$ at $\omega_n$. The cut-off function $\rho_n$ is smooth and such that $\rho_n(x)=\rho(x/\log(1+\sigma_n))$ for any $x\in\R^d$, with $0\le\rho\le 1$, $\rho\equiv 1$ on $B_1$ and $\mbox{supp }\rho\subset B_2$. Locally around~$(0,\omega_n)$, $\Pi_n$ converges to the identity while its first and second derivatives converge to $0$. Hence we know that
\[
\nrm{\nabla g_n}2^2\le\nrmcndn{\nabla f_n}2^2(1+o(1))\quad\mbox{and}\quad \nrmcndn{\nabla f_n}2^2\sim\sigma_n^{d-2}\,M_n^{-2}\,\nrmcnd{\nabla v_n}2^2
\]
as $n\to\infty$. Altogether, we find that
\[
\nrm{\nabla g_n}2^2=O\Big(t_n^{(\theta-\theta_n)\,\frac p{p-2}}\,L_n^{\frac{d-2}2-\frac 2{p-2}}\Big)
\]
which means that $(\nabla g_n)_n$ is bounded in $\L^2(\R^d)$ and $(g_n)_n$ converges in $\H^1_{\rm loc}(\R^d)$ to a constant if $L=\infty$. In any case, up to the extraction of a subsequence, $(g_n)_n$ converges weakly in $\H^1_{\rm loc}(\R^d)$. Since $\Delta g_n$ is bounded in $L^\infty(\R^d)$, by elliptic estimates, $(g_n)_n$ strongly converges in $C^2_{\rm loc} (\R^d)$ to a nonnegative function $g:\R^d\to \R$ such that, with $\theta=\vartheta(p,d)$,
\[\label{limiteq1}
-\,\theta\, \Delta g +A\,g= g^{p-1}\quad\mbox{in}\;\R^d\,,\quad g(0)=\nrm g\infty=1\;,
\]
where $g$ is constant and $A=0$ if $L=\infty$, and $A=(1-\theta)/L$ otherwise. However, if $L=\infty$, then $g\equiv1$ cannot be a solution. This proves that $L$ is finite and $A$ takes a finite, positive value. Moreover, $\nrm{g_n}2^2\sim\sigma_n^2\,t_n^{-1}\,\nrm{\nabla g_n}2^2\sim \nrm{\nabla g_n}2^2$ is bounded so that $(g_n)_n$ weakly converges in $\H^1(\R^d)$ to $g\not\equiv 0$. Hence we get
\begin{multline*}\hspace*{-12pt}
\liminf_{n\to\infty}\mathcal E_{\theta_n}[v_n]=\liminf_{n\to\infty}\;(t_n+\Lambda)^{\theta_n-\theta}\,\mathcal E_{\theta}[v_n]\\
\geq \liminf_{n\to\infty}\mathcal E_{\theta}[v_n]=\liminf_{n\to\infty}\frac{M_n^2}{\sigma_n^{2\,d/p}}\,\nrmcndn{\nabla f_n}2^{2\,\theta}\,\nrmcndn {f_n}2^{2\,(1-\theta)}\geq\frac{\nrmrd{\nabla g}2^{2\,\theta}\,\nrmrd {g}2^{2\,(1-\theta)}}{\lim_{n\to\infty}\nrm{f_n}p^2}\,,\hspace*{-3pt}
\end{multline*}
eventually after extraction of a subsequence, where the latter inequality holds by semi-continuity.

Let
\[
\mathcal E_{\theta_n, \mathcal C_n}[f]:= \sigma_n^{2\,(\vartheta_n-\vartheta)}\, \(\nrmcnd{\nabla f_n}2^2\!+\frac\Lambda{\sigma_n^2}\,\nrmcnd {f_n}2^2\)^{\theta_n}\nrmcnd {f_n}2^{2\,(1-\theta_n)}
\]
so that $\mathcal E_{\theta_n}[v_n]=\mathcal E_{\theta_n, \mathcal C_n}[f_n]\,/\,\nrmcndn{f_n}p^2$. Because of the change of variables, we know that Inequality \eqref{Ineq:Gen_interp_Cylinder} becomes
\be{Ineq:InterpScal}
\frac{\mathcal E_{\theta_n, \mathcal C_n}[f]}{\nrmcndn{f}p^2}\ge \frac 1{\C{CKN}(\vartheta_n,p,a)}\quad\forall\;f\in\H^1(\mathcal C_n)\;.
\ee

By the local strong convergence of the sequence $(g_n)_n$, there exists a sequence $(R_n)_n$ with $\lim_{n\to\infty}R_n=\infty$ such that
\[
\lim_{n\to\infty}\frac{\int_{\mathcal C_n\cap B_{R_n}} f_n^p\;dy}{\nrmcndn{f_n}p^p}=\delta\quad\mbox{and}\quad\lim_{n\to\infty}\frac{\int_{\mathcal C_n\cap (\R^d\setminus B_{4R_n})} f_n^p\;dy}{\nrmcndn{f_n}p^p}=1-\delta\;.
\]
Here $B_R$ denotes the ball of radius $R$ centered at $(0,\omega_n)$ in $\R^{d+1}$. Consider now two smooth cut-off functions $\rho$ and $\zeta$ such that $0\le\rho\le 1$, $0\le\zeta\le 1$, $\rho\equiv 1$ on $B_1$, $\zeta\equiv 1$ on $\R^{d+1}\setminus B_2$, and $\mbox{supp } \rho \subset B_2$, $\mbox{supp } \zeta \subset \R^{d+1}\setminus B_2$. Then we define $\rho_n$ and $\zeta_n$ by $\rho_n (x):=\rho(x/R_n)$ and $\zeta_n(x):=\zeta(x/R_n)$ for any $x\in\R^{d+1}$. We can write
\[
\frac1{\C{CKN}(\theta_n,p,a)}= \frac{\mathcal E_{\theta_n, \mathcal C_n}[f_n]}{\nrmcndn{f_n}p^2}\geq\frac{\mathcal E_{\theta_n, \mathcal C_n}[f_n\,\rho_n]+\mathcal E_{\theta_n, \mathcal C_n}[f_n\,\zeta_n]-\eta_n}{\nrmcndn{f_n}p^2}
\]
where $\eta_n=C/R_n$ for some constant $C>0$. Inequality~\eqref{Ineq:InterpScal} applied to $f_n\,\rho_n$ and $f_n\,\zeta_n$ shows that
\[
\frac1{\C{CKN}(\theta_n,p,a)}\geq\frac{\nrmcndn{f_n\,\rho_n}p^2+\nrmcndn{f_n\,\zeta_n}p^2-\eta_n}{\C{CKN}(\theta_n,p,a)\,\nrmcndn{f_n}p^2}\;.
\]
By passing to the limit $n\to\infty$, we find that $\delta\in(0,1]$ is such that
\[
\delta^{2/p}+(1-\delta)^{2/p}\le 1\;.
\]
Hence $\delta=1$, $\nrm gp^p=\lim_{n\to\infty}\nrmcndn{f_n\,\rho_n}p^p$ and using \eqref{Semi-Continuity}, we readily find that
\begin{multline*}
\frac 1{\C{CKN}(\vartheta(p,d),p,a)}\ge \liminf_{n\to\infty} \frac1{\C{CKN}(\theta_n,p,a)}=\liminf_{n\to\infty} \mathcal E_{\theta_n}[v_n]\\
\ge\frac{\nrmrd{\nabla g}2^{2\,\theta}\,\nrmrd {g}2^{2\,(1-\theta)}}{\nrm gp^2}\geq \frac 1{\C{GN}(p)}\;,
\end{multline*}
a contradiction with our hypothesis. \end{proof}

With Lemma~\ref{Lem:APriori2}, it is straightforward to establish the results of Theorem~\ref{MainThm2} using Proposition~\ref{Prop:Convergence1} as in Section~\ref{Sec:estimatesCKN}. Details are left to the reader.

\subsection{Compactness of sequences of extremal functions for the weighted logarithmic Hardy inequality approaching the limit case $\gamma= d/4$}\label{Sec:estimatesLH2}

\begin{lemma}\label{Lem:APriori4} Let $d\ge 3$, $a\in(-\infty, a_c)$, and assume that $\C{LS}<\C{WLH}(d/4,a)$. Consider a sequence $(\gamma_n)_n$ such that $\gamma_n>d/4$, $\lim_{n\to\infty}\gamma_n=d/4$ and a sequence $(w_n)_n$ of extremal functions in $\H^1(\mathcal C)$ for \eqref{Ineq:GLogHardy-w2} written for $\gamma=\gamma_n$: $\mathcal F_{\gamma_n}[w_n]=1/\C{WLH}(\gamma_n,a)$ and $\nrm{w_n}2=1$ for any $n\in\N$. Then $(w_n)_n$ is bounded in $\H^1(\mathcal C)$ if $\C{LS}<\C{WLH}(d/4,a)$. In that case, $(w_n)_n$ converges, up to translations and the extraction of a subsequence, towards a minimizer $w\in\H^1(\mathcal C)$ of $\mathcal F_{d/4}$, under the constraint $\nrmcnd w2=1$.\end{lemma}
\begin{proof} For any $n\in\N$, the function $w_n$ solves the Euler-Lagrange equation
\[\label{Eq:ELLH}
-\Delta w_n - \frac 1{2\,\gamma_n}\({\icnd{|\nabla w_n|^2}+\Lambda}\)w_n\,\big(1+\log |w_n|^2\big)=\mu_n\,w_n
\]
for some Lagrange multiplier $\mu_n\in \R$. Multiplying this equation by $w_n$ and integrating by parts, we get
\[\label{Eq:energy1}
\icnd{|\nabla w_n|^2}-\frac 1{2\,\gamma_n}\({\icnd{|\nabla w_n|^2}+\Lambda}\)\icnd{|w_n|^2\,\big(1+\log |w_n|^2\big)}\, =\mu_n\;.
\]
Let $t_n:=\nrmcnd{\nabla w_n}2^2$ and assume by contradiction that $\lim_{n\to\infty}t_n=\infty$. By definition of $\C{WLH}(\gamma,a)$, we know that
\[
\frac1{2\,\gamma_n}\,\icnd{|w_n|^2 \, \log |w_n|^2 }= \log\Big(\C{WLH}(\gamma_n,a)\,(t_n+\Lambda)\Big)\;.
\]
From the above estimates, we deduce that
\[\label{Eq:valuemun}
\mu_n=t_n -(t_n+\Lambda)\left[\frac1{2\,\gamma_n}+\log\Big(\C{WLH}(\gamma_n,a)\,(t_n+\Lambda)\Big)\right]
\]
also diverges as $n\to\infty$ like $-\,t_n\,\log t_n$. As in Section \ref{Sec:estimatesCKN2}, notice that, using approximate minimizers for the case $\gamma=d/4$, it is easy to verify that
\[
\liminf_{n\to\infty}\frac 1{\C{WLH}(\gamma_n,a)}\le\frac 1{\C{WLH}(d/4,a)}\;.
\]

Let us define $M_n=: \max_{\mathcal C} w_n$. By the maximum principle, we have
\[\label{max-est}
-\frac{t_n+\Lambda}{2\,\gamma_n}\,(1+\log M_n^2)\leq \mu_n\;,
\]
which shows that $M_n\ge t_n^{\gamma_n\,(1+o(1))}\to\infty$ as $n\to\infty$ and $a_n:=M_n^{1/\gamma_n}\,t_n^{-1}$ is such that $\liminf_{n\to\infty}a_n\ge 1$. Let $\sigma_n:=M_n^{2/d}$ and consider the sequence of rescaled functions $(f_n)_n$ defined on $\mathcal C_n:= \R\times \sigma_n\,\S$ by $w_n(\cdot)=M_n\, f_n(\sigma_n \,\cdot)$. These functions are such that $\nrmcndn{f_n}2=1$ and they solve
\[
-\,\Delta f_n-\frac{t_n+\Lambda}{2\,\gamma_n\,\sigma_n^2}\,f_n\,\log |f_n|^2=\left[\frac{\mu_n}{\sigma_n^2} +\frac{t_n+\Lambda}{2\,\gamma_n\sigma_n^2}\,\(1+\log\sigma_n^d\)\right]\,f_n\;.
\]
Moreover, we can assume with no restriction that the function $f_n$ attains its maximum value, $1$, at the point $(0,\omega_n)$ with $\omega_n=\sigma_n\,\omega_*$, for some given $\omega_*\in\S$. By assumption, we know that $\gamma_n\ge d/4$, so that, for $n$ large enough,
\[
\icn{|\nabla f_n|^2}=\frac{t_n}{\sigma_n^2}\le\frac{t_n}{\sigma_n^{d/(2\,\gamma_n)}}=t_n\,M_n^{-1/\gamma_n}=\frac 1{a_n}
\]
and
\[\label{aa4}
\frac{\mu_n}{\sigma_n^2} +\frac{t_n+\Lambda}{2\,\gamma_n\sigma_n^2}\,\(1+\log\sigma_n^d\)= \frac{t_n\, \log a_n}{\sigma_n^2}(1+o(1))\leq\frac{\log a_n}{a_n}\,(1+o(1))\;.
\]
As in Section~\ref{Sec:estimatesCKN2}, let $\Pi_n$ be the stereographic projection of $\sigma_n\,\S$ onto $\R^{d-1}$, considered as the tangent plane to $\S$ at $\omega_n$ where $\omega_n=\sigma_n\,\omega_*$. Let $g_n$ be such that $g_n(s,\Pi_n\,\omega)=f_n(s,\omega)\,\rho_n(s,\Pi_n\,\omega) $ for any \hbox{$(s,\omega)\in\R\times\sigma_n\,\S=\mathcal C_n$}. Here $\rho_n$ is a cut-off function as in Section~\ref{Sec:estimatesCKN2}. From the equation for $f_n$, we deduce that $\Delta g_n$ is bounded in $L^\infty(\R^d)$ uniformly with respect to $n\in\N$. Using elliptic estimates, up to the extraction of subsequences, we can prove that $(g_n)_n$ locally converges towards a function $g$, defined on $\R^d$ and satisfying
\[\label{aa5}
-\Delta g -A\,g\,(1+\log |g|^2)= B\,g\quad \mbox{in}\;\R^d\,,\quad g(0)=1\;.
\]
where $A=\frac 2d\,\lim_{n\to\infty}(t_n+\Lambda)/\sigma_n^2$ and $B:=\lim_{n\to\infty} (\log a_n)/a_n$ are two nonnegative real numbers. If $\lim_{n\to\infty} a_n = +\infty$, then $A=0=\nrm g2^2$, $B=0$ and $g$ satisfies $-\Delta g=0$ on $\R^d$, which means $g\equiv 1$. But on the other hand, the uniform boundedness of $f_n$ in $\L^2(\mathcal C_n)$ implies that $g\in\L^2(\R^d)$, a contradiction. Notice indeed that
\[
\ird{|g|^2}\le\liminf_{n\to\infty}\ird{|g_n|^2}\le\liminf_{n\to\infty}\icn{|f_n|^2}=1\;.
\]
The sequence $(a_n)_n$ is therefore bounded, $A$, $B$ are positive and $\delta:=\ird{|g|^2}\in(0,1]$. Notice indeed that $g\equiv 0$ would contradict $g(0)=1$ and hence $\delta>0$.

As a consequence of the strong convergence of $(g_n)_n$ in $\L^2_{\rm loc}(\R^d)$, there exists a sequence $(R_n)_n$ with $\lim_{n\to\infty}R_n=\infty$ such that

\[\label{BBB3}
\lim_{n\to\infty}\int_{\mathcal C_n\cap B_{R_n}} |f_n|^2\;dy= \delta\quad\mbox{and}\quad\lim_{n\to\infty}\int_{\mathcal C_n\cap(\R^{d+1}\setminus B_{4R_n})} |f_n|^2\;dy\geq 1-\delta\;.
\]
Here $B_R$ denotes the ball of center $(0,\omega_n)$ and radius $R$ in $\R^{d+1}\supset\mathcal C_n$. As in Section~\ref{Sec:estimatesCKN2}, consider two smooth cut-off functions $\rho$, $\zeta$, such that $\rho\equiv 1$ on $B_1$, $\zeta\equiv 1$ on $\R^d\setminus B_2$ and $\mbox{supp }\rho\subset B_2$, $\mbox{supp }\zeta\subset\R^d\setminus B_2$. Then, we define $\rho_n (x):=\rho(x/R_n)$, $\zeta_n(x):=\zeta(x/R_n)$. We know that $\nrmcndn{f_n}2^2\ge\nrmcndn{f_n\,\rho_n}2^2+\nrmcndn{f_n\,\zeta_n}2^2$ for any $n\in\N$, $\lim_{n\to\infty}\nrmcndn{f_n\,\rho_n}2^2=\delta$ and $\lim_{n\to\infty}\nrmcndn{f_n\,\zeta_n}2^2=1-\delta$. Moreover, we have $\nrmcndn{\nabla f_n}2^2\ge\nrmcndn{\nabla(f_n\,\rho_n)}2^2+\nrmcndn{\nabla (f_n\,\zeta_n)}2^2+\eta_n$ with $\eta_n=O(1/R_n)$.

For any $f\in\H^1(\mathcal C_n)$, define
\[
\mathcal F_{\gamma_n,\mathcal C_n}[f]:=\sigma_n^{2-\frac{d}{2\,\gamma_n}}\,\frac{\icn{|\nabla f|^2}+\frac\Lambda{\sigma_n^2}\icn{|f|^2}}{\exp\Big[\frac1{2\,\gamma_n}\icn{\frac{|f|^2}{\nrmcndn f2^2}\log\big(\frac{|f|^2}{\nrmcndn f2^2}\big)}\Big]}\,.
\]
Inequality \eqref{Ineq:GLogHardy-w2} simply amounts to $\mathcal F_{\gamma_n,\mathcal C_n}[f]\ge \frac{\icn{|f|^2}}{\C{WLH}(\gamma_n,a)}$ for any $f\in\H^1(\mathcal C_n)$. By assumption, we know that, for any $n\in\N$, $1/\C{WLH}(\gamma_n,a)=\mathcal F_\gamma[w_n]=\mathcal F_{\gamma_n,\mathcal C_n}[f_n]$ and \hbox{$\nrmcndn{f_n}2=1$}. From the above estimates, we have
\begin{multline*}
\mathcal F_{\gamma_n,\mathcal C_n}[f_n]\,\exp\(\frac1{2\,\gamma_n}\icn{|f_n|^2\,\log|f_n|^2}\)\\
=\nrmcndn{\nabla f_n}2^2+\Lambda\,\nrmcndn{f_n}2^2\ge\alpha_n+\beta_n+\eta_n
\end{multline*}
with
\[\begin{array}{l}
\alpha_n:=\nrmcndn{\nabla(f_n\,\rho_n)}2^2+\Lambda\,\nrmcndn{f_n\,\rho_n}2^2\;,\\[6pt]
\beta_n:=\nrmcndn{\nabla (f_n\,\zeta_n)}2^2+\Lambda\,\nrmcndn{f_n\,\zeta_n}2^2\;,\\[6pt]
\lim_{n\to\infty}\eta_n=0\;.
\end{array}\]
By definition of $\mathcal F_{\gamma_n,\mathcal C_n}$, we can rewrite $\alpha_n$ and $\beta_n$ as
\[\begin{array}{l}
\alpha_n=\mathcal F_{\gamma_n,\mathcal C_n}[f_n\,\rho_n]\left[\exp\(\!{\tfrac1{2\,\gamma_n}\icn{|f_n\,\rho_n|^2\,\log\Big(\tfrac{|f_n\,\rho_n|^2}{\nrmcndn{f_n\,\rho_n}2^2}\Big)\!}}\)\right]^{\nrmcndn{f_n\,\rho_n}2^{-2}},\\[12pt]
\beta_n=\mathcal F_{\gamma_n,\mathcal C_n}[f_n\,\zeta_n]\left[\exp\(\!{\tfrac1{2\,\gamma_n}\icn{|f_n\,\zeta_n|^2\,\log\Big(\tfrac{|f_n\,\zeta_n|^2}{\nrmcndn{f_n\,\zeta_n}2^2}\Big)\!}}\)\right]^{\nrmcndn{f_n\,\zeta_n}2^{-2}}.
\end{array}\]
By applying \eqref{Ineq:GLogHardy-w2} to $f_n\,\rho_n$ and $f_n\,\zeta_n$, we find that
\[
\mathcal F_{\gamma_n,\mathcal C_n}[f_n\,\rho_n]\ge\frac{\nrmcndn{f_n\,\rho_n}2^2}{\C{WLH}(\gamma_n,a)}\quad\mbox{and}\quad\mathcal F_{\gamma_n,\mathcal C_n}[f_n\,\zeta_n]\ge\frac{\nrmcndn{f_n\,\zeta_n}2^2}{\C{WLH}(\gamma_n,a)}\;.
\]
Using \cite[Theorem 2]{MR699419} and \eqref{CCC45}, we obtain
\begin{multline*}
\lim_{n\to\infty}\icn{|f_n|^2\,\log|f_n|^2}\\=\lim_{n\to\infty}\icn{|f_n\,\rho_n|^2\,\log|f_n\,\rho_n|^2}+\lim_{n\to\infty}\icn{|f_n\,\zeta_n|^2\,\log|f_n\,\zeta_n|^2}\;.
\end{multline*}
With
\[
\begin{array}{l}\x=\lim_{n\to\infty}\,\exp\left[\tfrac1{2\,\gamma_n}\icn{|f_n\,\rho_n|^2\,\log\Big(\tfrac{|f_n\,\rho_n|^2}{\nrmcndn{f_n\,\rho_n}2^2}\Big)}\right]\,,\\[12pt]
\y=\lim_{n\to\infty}\,\exp\left[\tfrac1{2\,\gamma_n}\icn{|f_n\,\zeta_n|^2\,\log\Big(\tfrac{|f_n\,\zeta_n|^2}{\nrmcndn{f_n\,\zeta_n}2^2}\Big)}\right]\,,
\end{array}\]
and $\eta=\delta$, if $\delta<1$, we find that
\begin{multline*}
\liminf_{n\to\infty}\frac 1{\C{WLH}(\gamma_n,a)}=\liminf_{n\to\infty}\mathcal F_{\gamma_n,\mathcal C_n}[f_n]\\
\ge \frac{\delta\,x^\frac 1\delta+(1-\delta)\,y^\frac 1{1-\delta}}{x\,y}\,\Big[\delta^\delta\,(1-\delta)^{1-\delta}\Big]^\frac 2d\,\liminf_{n\to\infty}\frac 1{\C{WLH}(\gamma_n,a)}\;.
\end{multline*}
Hence we know that $\delta^\delta\,(1-\delta)^{1-\delta}\le1$ by Lemma \ref{Lemma:algebr} (ii). This proves that $\delta\in(0,1]$ is actually equal to $1$.

\medskip Since $\C{WLH}(d/4,a)^{-1}\ge\liminf_{n\to\infty}\C{WLH}(\gamma_n,a)^{-1}$, we have
\begin{multline*}
\frac1{\C{WLH}(d/4,a)}\ge\lim_{n\to\infty}\frac1{\C{WLH}(\gamma_n,a)}=\liminf_{n\to\infty} \mathcal F_{\gamma_n}[w_n]\\
\hspace*{2cm}\geq \liminf_{n\to\infty} \mathcal F_{d/4}[w_n] =\liminf_{n\to\infty} \frac{\sigma_n^2\,\nrmcndn{\nabla f_n}2^2 +\Lambda} {\sigma_n^\frac d{2\,\gamma_n}\,e^{\frac{2}{d}\int_{\mathcal C_n} |f_n|^2\log |f_n|^2\;dy }}\\
\ge\frac{\ird{|\nabla f|^2} } {e^{\frac{2}{d}\int_{\R^d} |f|^2\log |f|^2\,dy }}\geq \frac 1{\C{LS}}\;,
\end{multline*}
a contradiction with the assumption $\C{LS}<\C{WLH}(d/4,a)$. This proves that $(t_n)_n$ is bounded.\end{proof}

With Lemma~\ref{Lem:APriori4}, it is straightforward to establish the results of Theorem~\ref {MainThm2} using Proposition~\ref{Prop:Convergence1} as in Section~\ref{Sec:estimatesLH}. Details are left to the reader. \red We postpone the proof the sufficient condition for $\C{LS}<\C{WLH}(d/4,a)$ to the next section. \nc

\section{Concluding remarks and open questions}\label{Sec:R}

Let us conclude with some comments on the range of the parameter $a$ for which \eqref{Ineq:GenInterp} admits extremal functions if $\theta=\vartheta(p,d)$. If \eqref{Cdt:Interp2} is satisfied, this is the case and because of the strict monotonicity of $a\mapsto\C{CKN}(\vartheta(p,d),p,a)$ as soon as $\C{GN}(p)<\C{CKN}(\vartheta(p,d),p,a)$, we know that this inequality also holds for any larger value of $a$, up to $a_c$. In Section~\ref{Sec:Main}, we gave a sufficient condition for which $\C{GN}(p)<\C{CKN}(\vartheta(p,d),p,a)$ holds. Let us give some details.

Consider $\mathcal R$ given by \eqref{Eqn:R}. To obtain $\mathcal R<0$, a sufficient condition is to have $\mathcal R_1<0$ and $\mathcal R_2<0$. This can be established in some cases.
\begin{proposition}\label{Lem:Refined} Let $d\ge 5$, $p\in(2,2^*)$ and $\theta=\vartheta(p,d)$. There is a constant $\bar a\in(-\infty, a_c)$ such that $\mathcal R$ is negative if $a\in(\bar a,a_c)$. In such a case, $\C{GN}(p)<\C{CKN}(\vartheta(p,d),p,a)$ holds and \eqref{Ineq:GenInterp} admits an extremal function in $\mathcal D^{1,2}_{a}(\R^d)$. \end{proposition}
Notice that the expressions of $\mathcal R_1$ and $\mathcal R_2$ being polynomial of order $2$ in $a$ and~$p$, an explicit expression of $\bar a$ can be established, which depends of $p$ and $d$.

\begin{proof} With no restrictions, that is, up to a scaling and a multiplication by a positive constant, the radial minimizer $u$ for $1/\mathsf C_{\rm GN}(p)$ solves the Euler-Lagrange equation
\[
-\Delta u+u-u^{p-1}=0\;.
\]
Let $\x_0:=\int_0^\infty|u'|^2\,r^{d-1}\,dr$, $\y_0:=\int_0^\infty|u|^2\,r^{d-1}\,dr$ and $\z_0:=\int_0^\infty|u|^p\,r^{d-1}\,dr$. Multiplying the equation by $u\,r^{d-1}$ and $r\,u'\,r^{d-1}$ and integrating with respect to $r\in(0,\infty)$, we find respectively
\[\label{Eq1}
\x_0+\y_0-\z_0=0
\]
and
\[\label{Eq2}
\tfrac{d-2}2\,\x_0+\tfrac d2\,\y_0-\tfrac dp\,\z_0=0\;.
\]
Let $\x_2:=\int_0^\infty|u'|^2\,r^{d+1}\,dr$, $\y_2:=\int_0^\infty|u|^2\,r^{d+1}\,dr$ and $\z_2:=\int_0^\infty|u|^p\,r^{d+1}\,dr$. Multiplying the equation by $u\,r^{d+1}$ and $r\,u'\,r^{d+1}$ and integrating with respect to $r\in(0,\infty)$, we find respectively
\[\label{Eq3}
\x_2-d\,\y_0+\y_2-\z_2=0
\]
and
\[\label{Eq4}
\tfrac{d-4}2\,\x_2+\tfrac{d+2}2\,\y_2-\tfrac{d+2}p\,\z_2=0\;.
\]
With $\mathsf e=y/|y|$, let us observe that
\begin{multline*}
|x+y|^{-2\,\gamma}=|y|^{-2\,\gamma}\,\(1+2\,\frac{x\cdot\mathsf e}{|y|}+\frac{|x|^2}{|y|^2}\)^{-\gamma}\\
=|y|^{-2\,\gamma}\,\(1-2\,\gamma\,\frac{x\cdot\mathsf e}{|y|}-\gamma\,\frac{|x|^2}{|y|^2}+2\,\gamma\,(\gamma+1)\,\frac{(x\cdot\mathsf e)^2}{|y|^2}+o\Big(\frac 1{|y|^2}\Big)\)
\end{multline*}
as $|y|\to\infty$. Consider a radial smooth function $g$, so that $\ird {x\,g}=0$, and define $g_n(x):=g(x+n\,\mathsf e)$. Using $\ird{(x\cdot\mathsf e)^2\,g}=\frac 1d\ird{|x|^2\,g}$, we find that
\[
\ird{|x|^{-2\,\gamma}\,g_n}=n^{-2\,\gamma}\left[1+\frac{\mathsf r(\gamma)}{n^2}\,\frac{\ird{|x|^2\,g}}{\ird g}+o\Big(\frac 1{n^2}\Big)\right]\quad\mbox{as}\;n\to\infty\;,
\]
where $\mathsf r(\gamma):=\tfrac 2d\,\gamma\,(\gamma-a_c)$. With the notations of Section~\ref{Sec:Main}, $\mathcal R$ given by \eqref{Eqn:R} takes the value
\[
\mathcal R=\theta\,\mathsf r(a)\,\frac{\x_2}{\x_0}+(1-\theta)\,\,\mathsf r(a+1)\,\frac{\y_2}{\y_0}-\frac 2p\,\,\mathsf r(\tfrac {b\,p}2)\,\frac{\z_2}{\z_0}\;.
\]
Using the above identities and $t:=\y_2/\y_0>0$, we can eliminate $\x_i$, $\y_i$, and $\z_i$ for $i=1$, $2$ in the expression of $\mathcal R=\mathcal R_1\,t+\mathcal R_0$ in terms of $t$. Notice that $\mathcal R_0$ and $\mathcal R_1$ are polynomials of degree two in terms of $a$, with finite coefficients depending on $p$, $d$ (but not on $t$). For $a=a_c$, we observe that
\[
\mathcal R=-\,\frac{d-4}{2\,p}\,\frac{2\,d-(d-2)\,p}{2\,(d+2)-(d-4)\,p}\,\big(2\,d+(p-2)\,t\big)\;,
\]
thus proving the result.\end{proof}

In practice, it turns out that the bound given by \eqref{Cdt:Interp2} is actually better than the condition of Proposition~\ref{Lem:Refined} in many cases. This however leaves open the question to decide if Inequality~\eqref{Ineq:GenInterp} with $\theta=\vartheta(p,d)$, $d\ge2$ and $p\in(2,2^*)$, admits extremal functions for any $a\in(-\infty,a_c)$ or if $\bar a:=\inf\{a\in(\bar a,a_c)\,:\,\C{GN}(p)<\C{CKN}(\vartheta(p,d),p,a)\}$ is finite. In such a case, \eqref{Ineq:GenInterp} would admit an extremal function for any $a>\bar a$ and would not admit any extremal function for any $a<\bar a$. If $\bar a>-\infty$, whether there is an extremal function for $a=\bar a$ is also open.

\red\medskip We finally provide a sufficient condition for having  $\C{LS}<\C{WLH}(\gamma,a)$, in order to prove the last statement of Theorem~\ref {MainThm2} (ii).
\begin{proposition}\label{Lem:RefinedWLH} Let $d\ge 3$. If $a\in(a_\star,a_c)$ with $a_\star$ as in Theorem~\ref{MainThm2}, then  $\C{LS}<\C{WLH}^*(d/4,a)$.
\end{proposition}
Details of the proof are left to the reader, as a consequence of \eqref{Ineq:CompRad} and $\C{LS}=2/(\pi\,d\,e)$. Whether the optimal interval is $(-\infty,a_c)$ or not is an open question. \nc

\begin{spacing}{0.5}\small\medskip\noindent{\small{\bf Acknowlegments.} This work has been partially supported by the projects EVOL and CBDif of the French National Research Agency (ANR).}

\noindent{\small \copyright\,2011 by the authors. This paper may be reproduced, in its entirety, for non-commercial purposes.}
\end {spacing}

\begin{spacing}{0.5}
\end{spacing}

\medskip\noindent Ceremade (UMR CNRS no. 7534), Universit\'e Paris-Dauphine, Place de Lattre de Tassigny, 75775 Paris Cedex 16, France.\\
J. Dolbeault: \textsf{dolbeaul@ceremade.dauphine.fr}\\
M.J. Esteban: \textsf{esteban@ceremade.dauphine.fr}

\end{document}